\documentclass[12pt, a4paper]{amsart}
\usepackage{amssymb, mathtools}
\usepackage{mathrsfs}
\usepackage{hyperref}

\newcommand{\Hom}{\operatorname{Hom}}
\newcommand{\D}{\mathscr{D}}
\newcommand{\DA}{\D_A}
\newcommand{\DB}{\D_B}
\newcommand{\dd}{\mathrm{d}}
\renewcommand{\Im}{\operatorname{Im}}

\theoremstyle{plain}
\newtheorem{theorem}{Theorem}[section]
\newtheorem{lemma}[theorem]{Lemma}

\newtheorem*{question}{Question}

\theoremstyle{definition}
\newtheorem{definition}[theorem]{Definition}

\theoremstyle{remark}
\newtheorem{example}[theorem]{Example}
\newtheorem{remark}[theorem]{Remark}
 
\begin{document}
\title[Differential operators
on nodal curves]{The ring of
differential operators on 
a nodal curve is not a 
bialgebroid}
\author{Myriam Mahaman}
\address{Faculty of Mathematics and Physics,
  Charles University in Prague,
  Sokolovská 49/83,
  Karlín,
  186 75 Praha 8,
  Czechia}

\email{myriam.mahaman@matfyz.cuni.cz}

\begin{abstract}
In a previous article, we showed that local projectivity
is a sufficient condition for the existence of a 
bialgebroid structure on the ring of differential operators
on an affine variety.
In this note, we show using elementary methods 
that the ring of differential operators on a nodal curve
is neither locally projective nor does it admit
a bialgebroid structure.
\end{abstract}

\subjclass{16S32, 16T15}

\maketitle

\section{Introduction}
Let \(A\) be a commutative algebra over a commutative ring  \(k\),
and  \(\DA\) be its ring of  \(k\)-linear differential operators
in the sense of Grothendieck
\cite{grothendieckElementsGeometrieAlgebrique1967}.
In the article \cite{krahmerRingDifferentialOperators2026}, 
we raised the following question:

\begin{question}
For which algebras \(A\) is  \(\DA\) a Hopf algebroid?
\end{question}

This question was motivated by Sweedler's results
in \cite{sweedlerGroupsSimpleAlgebras1974},
where he showed that
\(\DA\) is a Hopf algebroid
when \(A\) is the coordinate ring of a smooth
affine variety.
In our work \cite{krahmerRingDifferentialOperators2026}
we identified a more general property,
called \(R\)-\emph{local projectivity}
\cite{vercruysseLocalUnitsLocal2006},
which yields a positive
answer to our question.
Furthermore, we showed that 
this property is satisfied when \(A\)
is the coordinate ring of a monomial curve,
thereby providing
the first examples of non-smooth
varieties whose rings of differential operators
are Hopf algebroids.

When the assumption about local projectivity is removed,
the answer to our question remains unclear.
In this short note, we provide the first 
explicit examples for which 
there is no bialgebroid structure:
we take \(A\) to be
the coordinate ring of an affine nodal 
curve, as studied in \cite{smithDifferentialOperatorsAffine1988},
and using elementary methods, we show not only that 
there is no bialgebroid structure on \(\DA\),
but also that \(\DA\) is not  
\(R(A)\)-locally projective. 

This note is structured as follows:
we introduce our examples in Section~\ref{sec:nodalcurves},
then we show that \(\DA\) is not  \(R(A)\)-locally projective
in Section~\ref{sec:localprojectivity},
and finally, we show that 
\(\DA\) does not admit a bialgebroid
structure in Section~\ref{sec:bialgebroid}.

\section*{Acknowledgements}
The author thanks Ulrich Krähmer for helpful comments 
on an earlier version of this manuscript. 
This work was supported by the GAČR/NCN grant 
\emph{Quantum Geometric Representation Theory and 
Non-commutative Fibrations} 24-11728K.

\section{Nodal curves}
\label{sec:nodalcurves}
We fix our notation for the remainder of this paper:
let \(k\) be an algebraically closed
field of characteristic zero;
we denote \(\otimes \coloneqq \otimes_{k}\) and
\(\Hom \coloneqq \Hom_{k}\).
Let \(B \coloneqq k[t]\), 
\(\DB \coloneqq k\langle t, \partial \rangle 
/ (\partial t - t \partial - 1)\)
be its ring of \(k\)-linear differential operators,
\(I \coloneqq  f \, k[t]\) be an ideal of \(B\) such that
\(f = f_{1} \cdots f_{r}\) for some 
irreducible polynomials 
\(f_{1}, \ldots, f_{r} \in B\) (\(r \ge 2\))
which are pairwise coprime.
We consider the subalgebra \(A \coloneqq  k + I\) of \(B\)
and its ring of differential operators \(\DA\).
The following result was shown in 
\cite{smithDifferentialOperatorsAffine1988}:

\begin{lemma}
\(\DA = k + I \DB \subseteq \DB\).
\end{lemma}
\begin{proof}
  See \cite[Proposition 4.4]{smithDifferentialOperatorsAffine1988}.
\end{proof}

We will use the following properties to show our main results:

\begin{lemma}
\label{lem:counterexamples}
\begin{enumerate}
\item For all \(D \in \DA\), we have  \(D(I) \subseteq I\).
\item There exists \(D \in \DA\) such that
 \(D(A) \subseteq I\) and  \(D(I) \not \subseteq I^{2}\).
\item There exists  \(D \in \DA\) such that 
  \(D(I^{2}) \not \subseteq I^{2}\).
\end{enumerate}
\end{lemma}
\begin{proof}
\begin{enumerate}
\item Let \(D \in \DA\).
  There exists  \(\lambda \in k\)
  and  \(D' \in \DB\)
  such that  \(D = \lambda + fD'\),
  hence \(D(a) = \lambda a + fD'(a) \in I\)
  for any  \(a \in I\).
\item Let \(0 \neq g \in I\). 
  There exists \(D' \in \DB\)
  such that \(D'(1) = 0\) and \(D'(g) = 1\).
  Then \(D = fD' \in \DA\),
  and we have that 
  \(D(A) \subseteq I \) 
  and \(D(g) = f \not \in I^{2}\).
\item Let  \(0 \neq g \in I^{2}\).
  There exists \(D' \in \DB\)
  such that  \(D'(g) = 1\).
  Then  \(D = fD' \in \DA\),
  and we have \(D(g) = f \not \in I^{2}\).
  \qedhere
\end{enumerate}
\end{proof}

\begin{example}
We assume \(A\) to be the coordinate ring of 
the nodal cubic \(y^{2} = x^{2}(x+1)\).
Under the identifications \(x \mapsto t^{2} - 1\)
and 
\(y \mapsto t(t^{2} - 1)\), we obtain
that \(A = k + I\)
where \(I = (t^{2} - 1)k[t]\).
Straightforward computations show that
the operator
\((t^{2}-1)\frac{\dd}{\dd t}\) 
satisfies the condition~(2) 
in Lemma~\ref{lem:counterexamples},
and that \((t^{2}-1) \frac{\dd^{2}}{\dd t^{2}}\)
satisfies condition~(3).
\end{example}

\section{Local projectivity}
\label{sec:localprojectivity}

The notion of \(R\)-local projectivity
was originally introduced 
in \cite{vercruysseLocalUnitsLocal2006}
for an arbitrary \(A\)-module  \(M\).
Let us formulate how it looks 
for the case \(M = \DA\):
we consider the map
\[
  \Psi \colon \DA \to \Hom(A \otimes \DA, \DA),
  \quad
  D \mapsto \Psi_{D},
\]
where \(\Psi_{D}(a \otimes D') = D(a) D'\).

\begin{definition}
  We say that \(\DA\) is \emph{ \(R(A)\)-locally projective}
if  for all \(D \in \DA\), we have \(D \in \Im(\Psi_{D})\),
  i.e. there exist \(a_{1}, \ldots, a_{n} \in A\) 
  and \(D_{1}, \ldots, D_{n} \in \DA\) such that
  \[
    D = \sum_{i = 1}^{n} D(a_{i}) D_{i}.
  \]
\end{definition}

\begin{lemma}
\label{lem:locproj}
Let \(D \in \DA\) such that \(D(A) \subseteq I\).
Then  \(E(I) \subseteq I^{2}\) for all \(E \in \Im(\Psi_{D})\).
\end{lemma}
\begin{proof}
Let \(a \in A\) and  \(D' \in \DA\).
We have  \(\Psi_{D}(a \otimes D') = D(a)D'\),
where \(D(a) \in I\) and  \(D'(I) \subseteq I\),
hence  \(D(a)D'(I) \subseteq I^{2}\) 
and the result follows.
\end{proof}

\begin{theorem}
\(\DA\) is not  \(R(A)\)-locally projective.
\end{theorem}
\begin{proof}
We know from Lemma~\ref{lem:counterexamples}~(2)
that there exists
\(D \in \DA\) such that  \(D(A) \subseteq I\)
and  \(D(I) \not \subseteq I^{2}\).
Then it follows from Lemma~\ref{lem:locproj}
that \(D \not \in \Im(\Psi_{D})\),
hence \(\DA\) is not \(R(A)\)-locally projective.
\end{proof}

\section{Bialgebroid structure}
\label{sec:bialgebroid}
Let us show that there can be no bialgebroid
structure on \(\DA\).
Here \(\DA\) is endowed with the canonical counit
 \(\varepsilon \colon \DA \to A\), \(D \mapsto D(1)\),
and we show that there can be no associated
comultiplication. 
See \cite{krahmerRingDifferentialOperators2026}
for the full definition.

We consider the maps
\begin{align*}
  \mu_{*} \colon & \DA \otimes_{A} \DA \to \Hom(A \otimes A, A), &
  D_{1} \otimes_{A} D_{2} & \mapsto \big( a \otimes b 
  \mapsto D_{1}(a)D_{2}(b) \big), \\
  \mu^{*} \colon & \DA \to \Hom(A \otimes A, A), &
  D & \mapsto \big( a \otimes b 
  \mapsto D(ab) \big) 
,\end{align*}
where the tensor product \(\DA \otimes_{A} \DA\)
is obtained via
the left \(A\)-module structure of \(\DA\).

\begin{lemma}
\label{lem:bialgebroid}
There can be no map \(\Delta \colon \DA \to \DA \otimes_{A} \DA\)
 such that
\[
  \mu^{*}(D) = \mu_{*}(\Delta(D))
\]
for all \(D \in \DA\).
\end{lemma}
\begin{proof}
On the one hand, we have \(\varphi(I \otimes I) \subseteq I^{2}\)
 for all \(\varphi \in \Im \mu_{*}\).
Indeed, for any \(D_{1}, D_{2} \in \DA\),
we have \(D_{i}(I) \subseteq I\) for \(i = 1,2\),
therefore
\[
  \mu_{*}(D_{1} \otimes_{A} D_{2})(a \otimes b) 
  = D_{1}(a)D_{2}(b) \in I^{2}
\]
for all \(a,b \in I\),
hence  \(\mu_{*}(D_{1} \otimes_{A} D_{2})(I \otimes I)
\subseteq I^{2}\).

On the other hand, we know from
Lemma~\ref{lem:counterexamples}~(3)
that there exists \(D \in \DA\)
such that  \(D(I^{2}) \not \subseteq I^{2}\).
It follows that
\(\mu^{*}(D)(I \otimes I) \not \subseteq I^{2}\),
hence \(\mu^{*}(D) \not \in \Im \mu_{*}\),
and there can be no map \(\Delta\)
satisfying the given property.
\end{proof}

\begin{theorem}
There is no bialgebroid structure
on  \(\DA\) for which the map \(D \mapsto D(1)\)
is the counit.
\end{theorem}
\begin{proof}
The axioms defining a bialgebroid require 
the comultiplication \(\Delta \colon \DA \to \DA \otimes_{A} \DA\)
to satisfy the property in Lemma~\ref{lem:bialgebroid};
however, as we have shown, there can be no such map.
\end{proof}

\begin{remark}
It remains unclear whether in general \(R\)-local projectivity
is necessary for the existence of a bialgebroid structure,
since the arguments we used to disprove 
the two properties are unrelated.
\end{remark}

\bibliographystyle{plainurl}
\bibliography{bibliography}
\end{document}